\documentclass[reqno]{amsart}
\usepackage{amssymb}
\usepackage{hyperref}

\title[ ]
{Pseudo semi B-Fredholm and  Generalized Drazin invertible operators Through Localized SVEP }

\author[ A. Tajmouati, M. Karmouni,  M.Abkari ]
{  A. Tajmouati, M. Karmouni,  M.Abkari }

\address{A. Tajmouati, M. Karmouni,   M.Abkari \, \newline
 Sidi Mohamed Ben Abdellah
 Univeristy
 Faculty of Sciences Dhar Al Mahraz Fez, Morocco.}
\email{abdelaziztajmouati@yahoo.fr}
\email{mohammed.karmouni@usmba.ac.ma}
\email{mbark.abkari@usmba.ac.ma}


\newtheorem{theorem}{Theorem}[section]
\newtheorem{definition}{Definition}[section]
\newtheorem{remark}{Remark}

\newtheorem{proposition}{Proposition}[section]
\newtheorem{corollary}{Corollary}[section]
\newtheorem{example}{Example}

\begin{document}
\maketitle
\begin{abstract}
In this paper, we define and study the pseudo upper and lower semi B-Fredholm of  bounded operators in a Banach space. In particular, we prove equality
up to $S(T)$ between  the left generalized Drazin spectrum and the pseudo upper semi B-Fredholm spectrum,  $S(T )$
 is the set where $T$ fails to have  the SVEP. Also, we prove that both of the left and the right generalized Drazin operators are invariant under additive commuting  power finite rank perturbations and some perturbations for the pseudo upper and lower semi B-Fredholm operators are given.   As applications, we investigate some  classes of operators as the supercyclic and multiplier operators.
Furthermore, we investigate the left and the right generalized Drazin invertibility of upper triangular operator matrices by giving sufficient conditions which assure  that the left and the right  generalized Drazin spectrum or the pseudo upper and lower semi B-Fredholm of an upper triangular operator matrices  is the union of its diagonal entries  spectra.
\end{abstract}
\section{Introduction}
Throughout, $X$ denotes a complex Banach space and $\mathcal{B}(X)$ denotes the Banach algebra of all bounded linear
operators on $X$, we denote by $T^*$, $N(T)$,  $R(T)$, $ R^{\infty}(T)=\bigcap_{n\geq0}R(T^n)$,  $K(T)$, $ H_0(T)$,  $\rho(T)$, $\sigma_{ap}(T)$,   $\sigma_{su}(T)$,  $\sigma_{p}(T)$ and $\sigma(T)$,
 respectively the adjoint,  the null space, the range, the hyper-range, the analytic core, the quasi-nilpotent part, the resolvent set, the approximate point spectrum, the surjectivity spectrum, the point spectrum  and the spectrum of $T$.\\

 Next, let $T\in\mathcal{B}(X)$, $T$ has the single
valued extension property at $\lambda_{0}\in\mathbb{C}$ (SVEP) if
for every  open neighborhood   $U\subseteq \mathbb{C}$ of
$\lambda_{0}$, the only  analytic function  $f: U\longrightarrow
X$ which satisfies
 the equation $(T-zI)f(z)=0$ for all $z\in U$ is the function $f\equiv 0$. $T$ is said to have the SVEP if $T$ has the SVEP for
 every $\lambda\in\mathbb{C}$. Obviously, every operator $T\in\mathcal{B}(X)$ has the SVEP at every $\lambda\in\rho(T)$, then $T$ and $T^*$ have the SVEP at every point of the boundary  $\partial( \sigma(T))$ of the spectrum. In particular, $T$ and $T^*$ have the SVEP at every isolated point of the spectrum.  We denote by $S(T)$ the open set of $\lambda\in\mathbb{C}$ where $T$ fails to have SVEP at $\lambda$, and we say that  $T$ has SVEP if $S(T)=\emptyset$. It is easy to see that  $S(T)\subset\sigma_{p}(T)$, see \cite{Aie, lau}.\\
An operator $T \in\mathcal{B}(X)$ is said to be decomposable if for any open covering ${U_1, U_2}$ of the complex
plane $\mathbb{C}$,  there are two closed T-invariant subspaces $X_1$ and $X_2$ of $X$ such that
$X_1 + X_2 = X$ and $\sigma(T |X_k)\subset U_k$,  $k=1, 2$.
Note that $T$ is decomposable  implies that  $T$ and $T^*$ have the SVEP.\\
 A bounded linear operator is called an upper semi-Fredholm (resp, lower semi-Fredholm) if $dim N(T)<\infty \mbox{ and } R(T) \mbox{ is closed }$
(resp, $codim R(T) <\infty$). $T$ is semi-Fredholm if it is a lower or upper semi-Fredholm. The index of a semi Fredholm operator $T$ is defined by $ind(T)= dim N(T)- codim R(T)$

$T$ is a Fredholm operator if is  lower and upper semi-Fredholm, and T is called a Weyl operator if it is a Fredholm of index zero.\\
The upper, lower and semi-Fredholm  spectra of $T$  are  closed and defined by : $$\sigma_{uF}(T)=\{\lambda\in \mathbb{C}:\>\> T-\lambda I \mbox{  is not an  upper semi-Fredholm operator}\}$$
$$\sigma_{lF}(T)=\{\lambda\in \mathbb{C}:\>\> T-\lambda I \mbox{  is not a   lower  semi-Fredholm operator}\}$$
$$\sigma_{sF}(T)=\{\lambda\in \mathbb{C}:\>\> T-\lambda I \mbox{  is not a   semi-Fredholm operator}\}$$
The essential  and Weyl spectra  of $T$  are closed and  defined  by : $$\sigma_{e}(T)=\{\lambda\in \mathbb{C}:\>\> T-\lambda I \mbox{  is not a  Fredholm operator}\}$$
$$\sigma_{W}(T)=\{\lambda\in \mathbb{C}:\>\> T-\lambda I\>\> \mbox{ is not  a Weyl operator}\}$$
Now, consider a class of operators, introduced and studied by Berkani et al. in a series of papers
which extends the class of semi-Fredholm operators \cite{B1, B2, B3, AZ}. For every $T \in \mathcal{B}(X)$ and a nonnegative integer $n$, let us denote by $T_n$ the restriction of $T$ to $R(T^n)$ viewed as a map from the space $R(T^n)$ into itself (we set $T_0 = T$).\\
 An operator $T\in \mathcal{B}(X)$ is said to be upper (lower) semi B-Fredholm,  if for some integer $n\geq0$ the range $R(T^n)$ is closed and $T_n$ is an upper (lower) semi-Fredholm operator. Moreover, if $T_n$ is a Fredholm operator, then $T$ is called a B-Fredholm operator. A semi B-Fredholm operator is an upper or a lower semi B-Fredholm operator. It is easily seen that every nilpotent operator, as well as any idempotent bounded operator is B-Fredholm.
The class of B-Fredholm operators contains the class of Fredholm operators as a proper subclass.\\
Let $T\in \mathcal{B}(X)$, according to \cite[Proposition 2.6]{B1}, $T$ is a B-Fredholm operator if and only if there exists $(X_1,X_2)$ a pair of $T$- invariant closed
subspaces of $X$, such that  $X=X_{1}\oplus X_{2}$ and $T=T_{1}\oplus T_{2}$ where  $T_{1}$ is Fredholm and $T_{2}$ is nilpotent. The  upper, lower and B-Fredholm spectra are  defined by
$$\sigma_{uBF}(T)=\{\lambda\in \mathbb{C}:\>\> T-\lambda I \mbox{ is not upper  B-Fredholm} \}.$$
$$\sigma_{lBF}(T)=\{\lambda\in \mathbb{C}:\>\> T-\lambda I \mbox{ is not lower  B-Fredholm} \}.$$
$$\sigma_{BF}(T)=\{\lambda\in \mathbb{C}:\>\> T-\lambda I \mbox{ is not  B-Fredholm} \}.$$

Also $T\in \mathcal{B}(X)$ is a B-Weyl operator if there exists $(X_1,X_2)$ a pair of $T$- invariant closed
subspaces of $X$, such that  $X=X_{1}\oplus X_{2}$ and $T=T_{1}\oplus T_{2}$ where $T_{1}$ is Weyl operator and $T_{2}$ is nilpotent.
The  B-Weyl spectrum is defined by $$\sigma_{BW}(T)=\{\lambda\in \mathbb{C}:\>\> T-\lambda I  \mbox{ is not  B-Weyl}\}.$$

Let $T\in\mathcal{B}(X)$, the ascent of $T$ is defined by $a(T)=min\{p\in\mathbb{N}: N(T^p)=N(T^{p+1})\}$, if such $p$ does not exists we let $a(T)=\infty$. Analogously the descent of $T$ is $d(T)=min\{q\in\mathbb{N}: R(T^q)=R(T^{q+1})\}$, if such $q$ does not exists we let $d(T)=\infty$ \cite{LT}. It is well known that
if both $a(T)$ and $d(T)$ are finite then $a(T)=d(T)$ and we have the decomposition $X=R(T^p)\oplus N(T^p)$ where $p=a(T)=d(T)$.\\
The  descent and ascent spectra  of $T\in \mathcal{B}(X)$ are defined by :
    $$\sigma_{des}(T)=\{\lambda\in\mathbb{C}, T-\lambda I \mbox{   has  not finite   descent} \}$$
    $$\sigma_{acc}(T)=\{\lambda\in\mathbb{C}, T-\lambda I\mbox{  has   not  finite  ascent}\}$$

Let $T\in \mathcal{B}(X)$,
 $T$ is said to be a Drazin invertible if there exists a positive integer $k$ and an operator  $S\in\mathcal{B}(X)$ such that  $$ST=TS, \,\,\,T^{k+1}S=T^k\,\, \,\,and\,\,  S^2T=S.$$
Which is also equivalent to  the fact that $T=T_1\oplus T_2$; where $T_1$ is invertible  and $T_2$ is nilpotent.
It is well known  that $T$ is Drazin invertible if it has a finite ascent and descent.\\
Define two sets  $LD(X)$ and $RD(X)$  as \cite{Mbe2} :
$$LD(X)=\{ T\in\mathcal{B}(X): a(T)< \infty \mbox{ and } R(T^{a(T)+ 1}) \mbox{ is closed }\};$$
$$RD(X)=\{ T\in\mathcal{B}(X): d(T)< \infty \mbox{ and } R(T^{d(T)+ 1}) \mbox{ is closed }\}.$$
An operator $T\in \mathcal{B}(X)$ is said to be left (resp. right)  Drazin invertible if $T\in LD(X)$ (resp. $T\in RD(X)$).  The left and right  Drazin invertible spectra are defined by:
 $$\sigma_{lD}(T)=\{\lambda\in\mathbb{C},\,\, T-\lambda I\notin LD(X) \}$$
 $$\sigma_{rD}(T)=\{\lambda\in\mathbb{C},\,\, T-\lambda I\notin RD(X)\}$$
 And  we have  \cite{Mbe2, Mul, AZ}:
 $$\sigma_{D}(T)=\sigma_{lD}(T)\cup \sigma_{rD}(T)$$


The concept of Drazin invertible operators has been generalized by Koliha \cite{K}. In fact, $T\in \mathcal{B}(X)$ is generalized Drazin invertible if and only if $0\notin acc(\sigma(T))$ (  $acc(\sigma(T))$ is the set of all  points of accumulation of $\sigma(T)$), which is also equivalent to the fact that $T=T_1\oplus T_2$  where $T_1$ is invertible  and $T_2$ is quasi-nilpotent.  The generalized Drazin invertible spectrum defined by
$$\sigma_{gD}(T)=\{\lambda\in \mathbb{C}:\>\> T-\lambda I \mbox{ is not  generalized Drazin invertible} \}$$
In \cite{KMB}, the authors introduced  and studied  a new
concept of left and right generalized Drazin inverse of bounded operators as a generalization of left and right Drazin invertible operators. In fact,  an operator $T\in\mathcal{B}(X)$ is said to be left generalized Drazin invertible if $H_0(T)$ is closed and complemented with a subspace $M$ in $X$ such that $ T(M)$ is closed which equivalent to $T =T_1\oplus T_2$ such that $T_1$ is left invertible and $T_2$ is quasi-nilpotent see \cite[Proposition 3.2]{KMB}.\\
An operator $T\in\mathcal{B}(X)$ is said to be right generalized Drazin invertible if $K(T)$ is closed and complemented with a subspace $N$ in $X$ such that $ N\subseteq H_0(T)$  which equivalent to $T =T_1\oplus T_2$ such that $T_1$ is right  invertible and $T_2$ is quasi-nilpotent see \cite[Proposition 3.4]{KMB}.
The  left  and right generalized Drazin spectra   of $T\in \mathcal{B}(X)$ are  defined by:
    $$\sigma_{lgD}(T)=\{\lambda\in\mathbb{C},\,\, T-\lambda I \mbox{  is  not left  generalized   Drazin} \}$$
    $$\sigma_{rgD}(T)=\{\lambda\in\mathbb{C},\,\, T-\lambda I \mbox{ is  not right  generalized  Drazin} \}$$
     From \cite{KMB}, we have:
    $$\sigma_{gD}(T)=\sigma_{lgD}(T)\cup \sigma_{rgD}(T)$$
    $$\sigma_{rgD}(T)\subset\sigma_{rD}(T)$$
     $$\sigma_{lgD}(T)\subset\sigma_{lD}(T)$$
    $$\lambda \in \sigma_{lgD}(T) \Longleftrightarrow \lambda \in acc(\sigma_{ap}(T))$$
    $$\lambda \in \sigma_{rgD}(T) \Longleftrightarrow \lambda \in acc(\sigma_{su}(T))$$

    This paper is organized as follows. In section 2 and 3, we introduce  and study the class of pseudo upper semi B-Fredholm and pseudo lower semi B-Fredholm, and we show that the  pseudo upper semi B-Fredholm and pseudo lower semi B-Fredholm spectra, for a bounded linear  operator on a Banach space, are compact in the complex plane. Also, we prove equality up to $S(T)$ between  the left generalized Drazin
    spectrum and the pseudo upper semi B-Fredholm spectrum.  Some applications are given in section 4. In section 5,  we prove that the left and the right  generalized Drazin spectra of an operator are invariant  under additive commuting power finite rank perturbations. Some sufficient conditions are given to assure  that  the pseudo upper semi B-Fredholm,  pseudo lower semi B-Fredholm and pseudo B-Fredholm are stable under additive commuting power finite rank and nilpotent  perturbations. In section 6, we investigate the left and the right generalized Drazin spectra of upper triangular operator matrices
     $M_C=\begin{pmatrix}
A & B \\
0 & C \\
\end{pmatrix}$,  where $A\in\mathcal{B}(X)$, $B\in\mathcal{B}(Y)$ and $C\in\mathcal{B}(Y,X)$. We prove that $\sigma_{*}(M_0)=\sigma_{*}(A)\cup\sigma_{*}(B)$,
 $\sigma_{*} \in\{\sigma_{lgD}, \sigma_{rgD} \}$ and we give sufficient  conditions on $A$ (resp. on $B$) which ensure  the equality
 $\sigma_{lgD}(M_C)=\sigma_{lgD}(A)\cup\sigma_{lgD}(B)$ (resp. $\sigma_{rgD}(M_C)=\sigma_{rgD}(A)\cup\sigma_{rgD}(B)$).

    \section{Preliminaries }
    More recently, B-Fredholm and B-Weyl operators were generalized to pseudo B-Fredholm and pseudo B-Weyl \cite{BO}, \cite{ZZ}. Precisely,
$T$ is a pseudo B-Fredholm operator if there exists $(X_1,X_2)$ a pair of $T$- invariant closed
subspaces of $X$, such that  $X=X_{1}\oplus X_{2}$ and $T=T_{1}\oplus T_{2}$ where $T_{1}=T_{\shortmid X_{1}}$ is a Fredholm  operator and $T_{2}=T_{\shortmid X_{2}}$ is a quasi-nilpotent operator. The  pseudo B-Fredholm spectrum is defined by $$\sigma_{pBF}(T)=\{\lambda\in \mathbb{C}:\>\> T-\lambda I \mbox{ is  not  pseudo B-Fredholm} \}.$$
An operator $T$ is a pseudo B-Weyl operator if there exists $(X_1,X_2)$ a pair of $T$- invariant closed
subspaces of $X$, such that  $X=X_{1}\oplus X_{2}$ and $T=T_{1}\oplus T_{2}$ where $T_{1}=T_{\shortmid X_{1}}$  is a Weyl operator and $T_{2}=T_{\shortmid X_{2}}$ is a  quasi-nilpotent operator. The  pseudo B-Weyl spectrum is defined by $$\sigma_{pBW}(T)=\{\lambda\in \mathbb{C}:\>\> T-\lambda I \mbox{ is  not   pseudo B-Weyl} \}.$$
$\sigma_{pBW}(T)$ and $\sigma_{pBF}(T)$ is not necessarily non empty. For example, the quasi- nilpotent operator has empty pseudo B-Weyl  and B-Fredholm spectrum. Evidently $\sigma_{pBF}(T)\subset \sigma_{pBW}(T)\subset\sigma(T)$.\\
In the following, we define the  pseudo upper  semi B-Fredholm, pseudo lower semi B-Fredholm and pseudo semi B-Fredholm of a bounded
operator as a generalization of  semi B-Fredholm and we give some fundamental results concerning these operators.
\begin{definition}
An operator $T\in \mathcal{B}(X)$ is said to be  pseudo upper semi B-Fredholm if  there exists  two $T$-invariant  closed subspaces $X_{1}$ and $X_{2}$ of $X$  such that $X=X_1\oplus X_2$ and $T_{\shortmid X_{1}}$ is upper semi-Fredholm operator and $T_{\shortmid X_{2}}$ is quasi-nilpotent.
\end{definition}

\begin{definition}
An operator $T\in \mathcal{B}(X)$ is said to be  pseudo lower semi B-Fredholm if  there exists  two $T$-invariant  closed subspaces $X_{1}$ and $X_{2}$ of $X$  such that $X=X_1\oplus X_2$ and $T_{\shortmid X_{1}}$ is lower semi-Fredholm operator and $T_{\shortmid X_{2}}$ is quasi-nilpotent.
\end{definition}
\begin{definition}
 We say that $T\in\mathcal{B}(X)$ is pseudo semi B-Fredholm if  $T$ is  pseudo lower semi B-Fredholm  or  pseudo upper semi B-Fredholm.
 \end{definition}
It is clear that  $T$ is a pseudo B-Fredholm  operator if and only if $T$  is a pseudo lower semi B-Fredholm operator   and  pseudo upper semi B-Fredholm operator .\\
The  pseudo upper semi  B-Fredholm, pseudo lower  semi B-Fredholm  and pseudo semi B-Fredholm spectra are defined by $$\sigma_{pBuF}(T)=\{\lambda\in \mathbb{C}:\>\> T-\lambda I \mbox{ is not  a pseudo upper semi B-Fredholm} \}$$
$$\sigma_{pBlF}(T)=\{\lambda\in \mathbb{C}:\>\> T-\lambda I \mbox{ is not  a pseudo lower semi B-Fredholm} \}$$
$$\sigma_{pBsF}(T)=\{\lambda\in \mathbb{C}:\>\> T-\lambda I\mbox{ is not  a pseudo semi B-Fredholm} \}$$
Therefore, $$\sigma_{pBF}(T)=\sigma_{pBuF}(T)\cup\sigma_{pBlF}(T)$$ and $$\sigma_{pBsF}(T)=\sigma_{pBuF}(T)\cap\sigma_{pBlF}(T)$$
It is easy to see that $T$ is pseudo upper semi B-Fredholm if and only if $T^*$ is  pseudo  lower semi  B-Fredholm. Then:
$$\sigma_{pBuF}(T)=\sigma_{pBlF}(T^*)   \mbox{  and  } \sigma_{pBF}(T)=\sigma_{pBF}(T^*)$$
$\sigma_{pBsF}(T)$, $\sigma_{pBuF}(T)$ and $\sigma_{pBlF}(T)$ are  not necessarily non empty. For example, the quasi-nilpotent operator has empty pseudo upper semi B-Fredholm, pseudo lower semi B-Fredholm  and pseudo semi B-Fredholm spectrum.

\begin{example}\label{zz}
Let $T_1$ be defined  on $l^2(\mathbb{N})$ by :
\begin{center}
$T_0(x_1,x_2,....)=(x_1,0,x_2,0,x_3,0,....)$
\end{center}
$T_0$ is injective with closed range of infinite-codimension.\\
Consider the operator $T_2$ defined on $l^2(\mathbb{N})$ as: $T_2(x_1,x_2,....)=(x_2/2, x_3/3,.....)$.\\$T_2$ is  a quasi-nilpotent operator.
We have $T=T_0\oplus T_2$ is a pseudo upper semi B-Fredholm operator. Note that $0\in\sigma_{pBuF}(T_0)$, but $0\notin\sigma_{pBlF}(T_0)$.
\end{example}

\begin{example}
Let $T_1$ be defined  on $l^2(\mathbb{N})$ by :
\begin{center}
$T_1(x_1,x_2,....)=(x_2,x_3,....)$
\end{center}
 is surjective, but not injective, then is a lower semi-Fredholm operator. Let $T=T_1\oplus T_2$,  $T_2$ be as in Example \ref{zz}. Then $T$ is a pseudo lower semi B-Fredholm operator.
\end{example}

\section{ The class of pseudo semi B-Fredholm  operators}
Denote the open disc  centered at $\lambda_{0}$ with radius $\varepsilon$ in $\mathbb{C}$ by $D(\lambda_{0},\epsilon)$ and $D^*(\lambda_{0},\epsilon)=D(\lambda_{0},\epsilon)\backslash\{\lambda_0\}$.\\
The following Theorem establishes that if $T$ is a pseudo semi B-Fredholm operator, then $\lambda I -T$ is
 semi Fredholm in  an open punctured neighborhood of $0.$
\begin{theorem}\label{oo}
Let $T\in\mathcal{B}(X)$  a pseudo  semi B-Fredholm, then there exists a constant $\epsilon>0$ such that
$\lambda I -T$ is  semi Fredholm for all $\lambda\in \mathbb{D}^{*}(0,\epsilon).$
\end{theorem}

\begin{proof}
If $T$ is a pseudo  semi B-Fredholm operator, then there exists two closed $T$-invariant subspaces $X_1$ and $X_2$ such that $X=X_1\oplus X_2; T_{|X_1}$ is semi Fredholm, $T_{|X_2}$ is quasi-nilpotent and $T=T_{|X_1}\oplus T_{|X_2}.$\\
 If $X_{1}=\{0\}$, then $T$ is quasi-nilpotent,  then for all $\lambda \neq 0$ $\lambda I - T$ is invertible, hence $T-\lambda I$ is semi Fredholm for all $\lambda\neq0.$\\
 If $X_{1}\neq\{0\}$, thus $T_{\shortmid X_{1}}$ is  semi Fredholm,  then there exists  $\epsilon > 0$ such that $(T-\lambda I)_{\shortmid X_{1}} $ is  semi Fredholm for all $\lambda\in D(0,\epsilon)$.\\
As $T_{|X_2}$  is quasi-nilpotent, then $(T-\lambda I)_{\shortmid X_{2}}$ is invertible for all $\lambda\neq0$, hence $(T-\lambda I)_{\shortmid X_{2}}$
is semi Fredholm.
 Therefore, $(T-\lambda I)_{\shortmid X_{1}}$ is semi Fredholm  for all $\lambda\in \mathbb{D}(0,\epsilon)$ and $(T-\lambda I)_{\shortmid X_{2}}$ is semi Fredholm for all $\lambda\neq0$, hence  $(T-\lambda I)$ is  semi Fredholm for all $\lambda\in \mathbb{D}^{*}(0,\epsilon).$
\end{proof}

From  Theorem  \ref{oo},  we derive the following corollary.
\begin{corollary}
Let $T\in\mathcal{B}(X)$, then $\sigma_{pBuF}(T), \sigma_{pBlF}(T), \sigma_{pBsF}(T)$ are compact subsets of $\mathbb{C}.$\\
Moreover $\sigma_{uF}(T)\setminus\sigma_{pBuF}(T)$,  $\sigma_{lF}(T)\setminus\sigma_{pBlF}(T)$,  $\sigma_{sF}(T)\setminus\sigma_{pBsF}(T)$
consist of at most countably many isolated points.
\end{corollary}
Since $\sigma_{pBuF}(T)\subset\sigma_{uBF}(T)\subset\sigma_{uF}(T)$ and $\sigma_{pBlF}(T)\subset\sigma_{lBF}(T)\subset\sigma_{lF}(T)$
the following corollary hold:
\begin{corollary}
Let $T\in\mathcal{B}(X)$, then  $\sigma_{uBF}(T)\setminus\sigma_{pBuF}(T)$, $ \sigma_{lBF}(T)\setminus\sigma_{pBlF}(T)$
consist of  at most countably many isolated points.
\end{corollary}
Recall that $T\in\mathcal{B}(X)$ is said to be Kato operator or
semi-regular if $R(T)$ is closed
and $N(T)\subseteq R^{\infty}(T)$. Denote by $\rho_{K}(T)$ :
$\rho_{K}(T)=\{\lambda\in\mathbb{C}: T-\lambda I\mbox{  is Kato }
\}$ the Kato resolvent  and
$\sigma_{K}(T)=\mathbb{C}\backslash\rho_{K}(T)$ the Kato spectrum
of $T$.
An operator $T\in \mathcal{B}(X)$ admit a generalized Kato decomposition, abbreviated as GKD if  there exists  two $T$-invariant  closed subspaces $X_{1}$ and $X_{2}$ of $X$  such that $X=X_1\oplus X_2$ and $T_{\shortmid X_{1}}$ is semi-regular (or a Kato) operator and $T_{\shortmid X_{2}}$ is quasi-nilpotent.
It is easy to see that every pseudo semi B-Fredholm is a pseudo Fredholm. According to \cite[Theorem 2.2]{JZ}, the following proposition hold.

\begin{proposition}
Let $T\in\mathcal{B}(X)$  a pseudo  semi B-Fredholm operator.
Then there exists a constant  $\epsilon > 0$ such that   $T-\lambda I$ is a Kato operator, for all $\lambda\in \mathbb{D}^{*}(0,\epsilon)$.
\end{proposition}
As a consequence of the preceding Proposition, we have:
\begin{corollary}
Let $T\in\mathcal{B}(X)$, $\sigma_{K}(T)\setminus\sigma_{pBsF}(T)$
consist of at most countably many isolated points.
\end{corollary}
Define   $pBuW(X)$  and  $pBlW(X)$ :\\
 $T\in pBuW(X)$  if there exists  two $T$-invariant  closed subspaces $X_{1}$ and $X_{2}$ of $X$ where $X=X_1\oplus X_2$ and $T_{\shortmid X_{1}}$ is upper semi Fredholm of $ind(T_{\shortmid X_{1}})\leq 0$  and $T_{\shortmid X_{2}}$ is quasi-nilpotent.\\
$T\in pBlW(X)$ if there exists  two $T$-invariant  closed subspaces $X_{1}$ and $X_{2}$ of $X$ where $X=X_1\oplus X_2$ and $T_{\shortmid X_{1}}$ is lower semi Fredholm of $ind(T_{\shortmid X_{1}})\geq 0$ and $T_{\shortmid X_{2}}$ is quasi-nilpotent. The corresponding spectra of these sets  are defined by:
$$\sigma_{pBuW}(T)=\{\lambda\in\mathbb{C}, T-\lambda I\notin pBuW(X)\}$$
$$\sigma_{pBlW}(T)=\{\lambda\in\mathbb{C}, T-\lambda I\notin pBlW(X)\}$$
Then $\sigma_{pBuF}(T)\subseteq\sigma_{pBuW}(T)\subseteq\sigma_{lgD}(T)$ and $\sigma_{pBlF}(T)\subseteq\sigma_{pBlW}(T)\subseteq\sigma_{rgD}(T)$.

\begin{remark}
We have $\sigma_{pBuF}(T)\subset\sigma_{lgD}(T)$, this inclusion is proper. Indeed: Let $T$ be the
unilateral left shift operator defined on the Hilbert $l^2(\mathbb{N})$:
$$ L(x_{1},x_{2},....)=(x_{2},x_{3},...).$$
  Since $\sigma_{ap}(L)=\{\lambda\in \mathbb{C}; |\lambda|\leq 1\}$, then $\sigma_{lgD}(L)=acc(\sigma_{ap}(L))=\{\lambda\in \mathbb{C}; |\lambda|\leq 1\}$  Observe that $L$ is an upper semi-Fredholm operator, then $0\notin \sigma_{pBuF}(T)$.
  This shows that the inclusion $\sigma_{pBuF}(T)\subset\sigma_{lgD}(T)$ is proper. Then it is naturel to ask about the defect set $\sigma_{lgD}(T)\backslash\sigma_{pBuF}(T)$, where $T$ is a bounded operator.
\end{remark}
In the following theorem, we get a characterization of this defect set.

\begin{theorem}\label{t11}
Let $T\in\mathcal{B}(X)$. Then: $$\sigma_{lgD}(T)=\sigma_{pBuF}(T)\cup S(T)=\sigma_{pBuW}(T)\cup S(T)$$
\end{theorem}
\begin{proof}
$S(T)\subseteq\sigma_{lgD}(T)$ and $\sigma_{pBuF}(T)\subseteq\sigma_{lgD}(T)$, hence $\sigma_{pBuF}(T)\cup S(T)\subseteq \sigma_{lgD}(T)$. Indeed:
Let $\lambda\notin\sigma_{lgD}(T)$ then $T-\lambda I$ is a left generalized Drazin invertible,  then $H_0(T-\lambda I)$ is closed,  by \cite[Theorem 1.7]{AP} $T$ has the SVEP at $\lambda$, hence $S(T)\subseteq\sigma_{lgD}(T)$.
Conversely : Let $\lambda \notin \sigma_{pBuF}(T)\cup S(T)$, then $T-\lambda I$ is a pseudo upper semi B-Fredholm, then there exists two closed subspaces $T$-invariant $X_1$
and $X_2$ such that $X=X_1\oplus X_2$, $(T-\lambda I)_{|X_1}$ is upper semi Fredholm and $(T-\lambda I)_{|X_2}$ is quasi-nilpotent. Since $T$ has SVEP at $\lambda$ then $T_{|X_1}$ and $T_{|X_2} $ have the  SVEP at $\lambda.$
Since  $(T-\lambda I)_{|X_1}$ is upper semi Fredholm, then $(T-\lambda I)_{|X_1}$ is upper semi B-Fredholm, since $T_{|X_1}$  has  SVEP at $\lambda,$
by \cite[Theorem 2.1]{AZ}, we have $(T-\lambda I)_{|X_1}$ is left Drazin invertible, hence $X_1=X_{1}^{'}\oplus X_{1}^{''}$
with $(T-\lambda I)_{|X_{1}^{'}}$ is invertible and $(T-\lambda I)_{|X_{1}^{''}}$ is nilpotent, thus $X=X_{1}^{'}\oplus X_{1}^{''}\oplus X_2$
with $(T-\lambda I)_{|X_{1}^{'}}$ is invertible and $(T-\lambda I)_{| X_{1}^{''}\oplus X_2}$ is quasi-nilpotent, therefore  $(T-\lambda I)$ is left generalized Drazin invertible.
\end{proof}
By duality we get a similar result for the right Drazin invertible spectrum.
\begin{theorem}\label{t22}
Let $T\in\mathcal{B}(X)$. Then:
    $$\sigma_{rgD}(T)=\sigma_{pBlF}(T)\cup S(T^*)=\sigma_{pBlW}(T)\cup S(T^*)$$
\end{theorem}
\begin{proof}
 $\sigma_{rgD}(T)=\sigma_{lgD}(T^*)=\sigma_{pBlF}(T^*)\cup S(T^*)=\sigma_{pBlF}(T)\cup S(T^*)$,
then $\sigma_{rgD}(T)=\sigma_{pBlF}(T)\cup S(T^*)$
\end{proof}
\begin{remark}
 Theorem \ref{t11}  and Theorem \ref{t22} extend \cite[Theorem 2.1]{AZ} and  \cite[Theorem 2.2]{AZ}.
 \end{remark}
From the preceding Theorems we get the following corollaries.
\begin{corollary}
Let $T\in\mathcal{B}(X)$. Then
    $$\sigma_{gD}(T)=\sigma_{pBF}(T)\cup S(T)\cup S(T^*)$$
\end{corollary}

\begin{corollary}\label{c1}
Let $T\in \mathcal{B}(X)$.\\
  If $T$ has  the SVEP then:
 $$\sigma_{lgD}(T)=\sigma_{pBuF}(T)=\sigma_{pBuW}(T)   \mbox{           :  (1) }$$
If $T^*$ has the SVEP then:
$$\sigma_{rgD}(T)=\sigma_{pBlF}(T)=\sigma_{pBlW}(T)    \mbox{           :  (2) }$$
if both $T$ and $T^*$ have the  SVEP then all the spectra  in $(1)$ and $(2)$ coincide  and are equal to the pseudo B-Fredholm, pseudo B-Weyl and generalized Drazin spectra.
\end{corollary}

\begin{example}
Let $T$ be the unilateral weighted shift on $l^2(\mathbb{N})$ defined by:
$$Te_n=\left\{
  \begin{array}{ll}
    0, &  \hbox{if}\,\,  n=p!\,\, for\,\, some\,\, p\in\mathbb{N}\\
    e_{n+1} &   \hbox{otherwise.}
  \end{array}
\right.$$
The adjoint operator of $T$ is :
$$T^*e_n=\left\{
  \begin{array}{ll}
    0 &  \hbox{if}\,\,  n=0\,\,or\,\, n=p!+1\,\, for\,\, some\,\, p\in\mathbb{N}\\
    e_{n-1} &   \hbox{otherwise.}
  \end{array}
\right.$$
We have $\sigma(T)=\overline{D(0,1)}$ the unit closed  disc.
The point spectrum of $T$ and $T^*$ are : $\sigma_{p}(T)=\sigma_{p}(T^*)=\{0\}$, hence $T$ and $T^*$ have the SVEP.
Then $\sigma_{ap}(T)=\sigma_{su}(T)=\sigma(T)$, hence $\sigma_{gD}(T)=\sigma_{lgD}(T)=\sigma_{rgD}(T)=\overline{D(0,1)}$
From  Corollary \ref{c1}, we have
 $$\sigma_{pBuF}(T)=\sigma_{pBlF}(T)=\sigma_{pBF}(T)=\sigma_{gD}(T)=\sigma_{lgD}(T)=\sigma_{rgD}(T)=\overline{D(0,1)}$$
 \end{example}


\section{Applications}
A bounded linear operator $T$ is called supercyclic
provided there is some $x \in X$ such that the set $\{\lambda T^n, \,\, \lambda\in\mathbb{C}\,\, ,n=0,1,2,..\}$
is dense in $X$. It is well Known that if $T$ is supercyclic then $\sigma_{p}(T^*)=\{0\}$ or $\sigma_{p}(T^*)=\{\alpha\}$
for some nonzero $\alpha\in\mathbb{C}$.  Since an operator with countable point spectrum has SVEP, then we have the following:
\begin{proposition}
Let $T\in\mathcal{B}(X)$, a supercyclic operator. Then :
\begin{center}
    $\sigma_{rgD}(T)=\sigma_{pBlF}(T)$
\end{center}
\end{proposition}
Since, Every hyponormal operator $T$ on a Hilbert space has the single valued extension property, we have
\begin{proposition}
Let $T$ a hyponormal operator on a Hilbert space then:
\begin{center}
    $\sigma_{lgD}(T)=\sigma_{pBuF}(T)$
\end{center}
In particular, If $T$ is auto-adjoint then :     $\sigma_{gD}(T)=\sigma_{pBF}(T)$
\end{proposition}

Let $A$ be a semi-simple commutative Banach algebra.\\
The mapping $T$ $:$ $A\longrightarrow  A$ is said to be a multiplier of $A$ if
$T (x)y = xT (y)$ for all $x, y \in A.$\\
It is well known each multiplier  on $A$  is a continuous  linear operator  and that the set of all multiplier on $A$ is a unital closed commutative  subalgebra of $B(A)$ \cite[Proposition 4.1.1]{lau}. Also
the semi-simplicity of A implies that every multiplier has the SVEP (see \cite[Proposition 2.2.1]{lau}).
According to  Corollary \ref{c1} we have
\begin{proposition}\label{pp}
Let $T$ be a multiplier on semi-simple commutative Banach
algebra $A$, then the following assertions are equivalent
\begin{enumerate}
  \item  $T$ is pseudo upper semi B-Fredholm .
  \item  $T$ is  left generalized Drazin invertible.
\end{enumerate}
\end{proposition}
Now if assume in additional that $A$ is regular and Tauberian (see \cite[Definition 4.9.7]{lau}), then every multiplier $T^*$ has  SVEP. Hence, we have the following Proposition.
\begin{proposition}\label{ppp}
Let $T$ be a multiplier on semi-simple regular and Tauberian commutative Banach
algebra $A$, then the following assertions are equivalent
\begin{enumerate}
  \item  $T$ is pseudo B-Fredholm .
  \item  $T$ is   generalized Drazin invertible.
\end{enumerate}
\end{proposition}
Let  $G$ a locally compact abelian group, with group operation + and Haar measure $\mu$,  let $L^1(G)$ consist of all $\mathbb{C}$-valued
functions on $G$ integrable with respect to Haar measure and $M (G)$ the Banach
algebra of regular complex Borel measures on $G$. We recall that $L^1(G)$ is a regular
semi-simple Tauberian commutative Banach algebra. Then we have the following:

\begin{corollary}\label{cc}
Let $G$ be a locally compact abelian group, $\mu \in M (G)$. Then every convolution operator $T_{\mu}$
$: L^1(G)\longrightarrow L^1(G)$, $T_{\mu}(k) = \mu \star k$ is pseudo B-Fredholm if and only if is   generalized Drazin invertible.
\end{corollary}

\begin{remark}
Proposition \ref{ppp} and corollary \ref{cc} generalize \cite[Proposition 3.4]{AZ} and \cite[Corollary 3.3]{AZ}. These results also generalize some results in \cite{B3}
\end{remark}

\section{Perturbation}
Let $\mathcal{F}(X)$ denote the ideal of finite rank operators on $X$.
In the following, we show that both $\sigma_{lgD}(T)$ and $\sigma_{rgD}(T)$ are stable under
additive commuting power finite rank operator.
\begin{proposition}\label{pa}
Suppose that $F\in\mathcal{B}(X)$ satisfies $F^n\in\mathcal{F}(X)$ for some $n \in \mathbb{N}$ and that
$T\in\mathcal{B}(X)$ commutes with $F$. Then we have
$$\sigma_{lgD}(T)=\sigma_{lgD}(T+F)  \mbox{  and   }\sigma_{rgD}(T)=\sigma_{rgD}(T+F)$$
\end{proposition}
\begin{proof}
According to \cite[ Theorem 2.2]{ZZY}, we have $acc(\sigma_{ap}(T))=acc(\sigma_{ap}(T+F))$. Then $\lambda\in  \sigma_{lgD}(T)$ if and only if $\lambda\in acc(\sigma_{ap}(T))$ if and only if $\lambda\in acc(\sigma_{ap}(T+F))$ if and only if $\lambda\in\sigma_{lgD}(T+F)$. So $\sigma_{lgD}(T+F)=\sigma_{lgD}(T)$.\\
By duality, we have $\sigma_{rgD}(T)=\sigma_{rgD}(T+F)$.
\end{proof}
As a consequence of  Proposition \ref{pa},  we have  the following corollary.
\begin{corollary}\label{ee}
Suppose that $F\in \mathcal{B}(X)$ satisfies $F^n\in\mathcal{F}(X)$ for some $n \in \mathbb{N}$ and that
$T\in\mathcal{B}(X)$ commutes with $F$. Then we have
$$\sigma_{gD}(T)=\sigma_{gD}(T+F)$$
\end{corollary}

The following example illustrates that  the approximate point spectrum  $\sigma_{ap}(.)$ in general is not preserved under
commuting finite rank perturbations.
\begin{example}
Let $A\in\mathcal{B}(l_2)$ defined by:
$$ A(x_{1},x_{2},....)=(0, x_{1},x_{2},...).$$

Let $0<\varepsilon< 1$,  $F_{\varepsilon}\in\mathcal{B}(l_2)$ be a finite rank operator  defined by:
$$ F_{\varepsilon}(x_{1},x_{2},....)=(- \varepsilon x_{1},0,0, ...).$$
Let  $T=A\oplus I$ and $F=0\oplus F_{\varepsilon}.$ Then $F$ is a finite rank operator and $TF=FT$. But
 $\sigma_{ap}(T)=\{\lambda\in\mathbb{C},  |\lambda|=1\}$,  $\sigma_{ap}(T+F)=\{\lambda\in\mathbb{C},  |\lambda|=1\}\cup\{1-\varepsilon \}$.
 \end{example}

Using  Corollary \ref{c1}, Proposition \ref{pa} and Corollary \ref{ee}, we can  prove the following corollary.
\begin{corollary}
Suppose that $F\in\mathcal{B}(X)$ satisfies $F^n\in\mathcal{F}(X)$ for some $n \in \mathbb{N}$ and that
$T\in\mathcal{B}(X)$ commutes with $F$.
\begin{enumerate}
  \item If $T$ has the SVEP, then $\sigma_{pBuF}(T)=\sigma_{pBuF}(T+F)$
  \item If $T^*$ has the SVEP, then $\sigma_{pBlF}(T)=\sigma_{pBlF}(T+F)$
  \item If $T$ and $T^*$ have the SVEP, then $\sigma_{pBF}(T)=\sigma_{pBF}(T+F)$
\end{enumerate}
\end{corollary}

Let $T\in \mathcal{B}(X)$, $Q$ a quasi-nilpotent such that $QT=TQ$, from \cite[Proposition 2.9]{ZZ},  we have
$\sigma_{lgD}(T+Q)=\sigma_{lgD}(T)$ and $\sigma_{rgD}(T+Q)=\sigma_{rgD}(T).$

\begin{proposition}\label{p0}
Let $T\in \mathcal{B}(X)$, $N$ a nilpotent operator commutes with $T$ then:
 $$\sigma_{pBuF}(T+N)\cup S(T)=\sigma_{pBuF}(T)\cup S(T)$$

\end{proposition}
\begin{proof}
From  Theorem \ref{t11}, we have  $\sigma_{lgD}(T)=\sigma_{pBuF}(T)\cup S(T),$ then $\sigma_{lgD}(T+N)=\sigma_{pBuF}(T+N)\cup S(T+N)$
since  $\sigma_{lgD}(T+N)=\sigma_{lgD}(T)$ and $S(T+N)=S(T)$, hence $\sigma_{pBuF}(T+N)\cup S(T)=\sigma_{pBuF}(T)\cup S(T)$.
\end{proof}
By duality, we have the following proposition.
\begin{proposition}\label{p00}
Let $T\in \mathcal{B}(X)$, $N$ a nilpotent operator commutes with $T$ then:
 $$\sigma_{pBlF}(T+N)\cup S(T^*)=\sigma_{pBlF}(T)\cup S(T^*)$$
\end{proposition}
As a consequence of Proposition \ref{p0} and Proposition \ref{p00}, the following corollary hold.
\begin{corollary}
Let $T\in \mathcal{B}(X)$, $N$ a nilpotent operator commute with $T$ then:
 $$\sigma_{pBF}(T+N)\cup S(T)\cup S(T^*)=\sigma_{pBF}(T)\cup S(T)\cup S(T^*)$$
\end{corollary}
\begin{corollary}
Let $T\in\mathcal{B}(X)$, $N$ a nilpotent operator commutes with $T$. If $T$ has  SVEP, then
    $$\sigma_{pBuF}(T+N)=\sigma_{pBuF}(T)$$
\end{corollary}
\begin{corollary}
Let $T\in\mathcal{B}(X)$, $N$ a nilpotent operator commutes with $T$. If $T^*$ has  SVEP, then
    $$\sigma_{pBlF}(T+N)=\sigma_{pBlF}(T)$$
\end{corollary}

\begin{corollary}
Let $T\in \mathcal{B}(X)$, $N$ a nilpotent operator commutes with $T$,  if $T$ and $T^*$ have  SVEP, then
 $$\sigma_{pBF}(T+N)=\sigma_{pBF}(T)$$
\end{corollary}




\section{Left and Right generalized Drazin invertibility for operator matrices}

Let $X$ and $Y$ be   Banach spaces and $B(X,Y)$ denote the space of all bounded  linear operator from $X$ to $Y$.\\
For $A\in\mathcal{B}(X)$, $B\in\mathcal{B}(Y)$, we denote by $M_C\in\mathcal{B}(X\oplus Y)$ the operator  defined on $X\oplus Y$ by
$$
\begin{pmatrix}
A & C \\
0 & B \\
\end{pmatrix}
$$
It is well know that, in the case of infinite dimensional, the inclusion $\sigma(M_C)\subset\sigma(A)\cup\sigma(B)$, may be strict.\\
This motivates serval authors to study the defect ($\sigma_{*}(A)\cup\sigma_{*}(B))\setminus \sigma_{*}(M_C)$ where $\sigma_{*}$ runs different  type spectra.\\
In this section we interested  and motivated by the relationship between $\sigma_{*}(M_C)$ and $\sigma_{*}(A)\cup\sigma_{*}(B)$, where
$\sigma_{*} \in\{\sigma_{lgD}, \sigma_{rgD}, \sigma_{gD}, \sigma_{pBuF}, \sigma_{pBlF}, \sigma_{pBF}\}$.

We start this section by proving that the left generalized Drazin spectrum of a direct sum  is  the union of left generalized Drazin spectra of its summands.
\begin{proposition}
Let $(A, B)\in \mathcal{B}(X)\times \mathcal{B}(Y)$  and  $C\in\mathcal{B}(Y, X)$. Then :
$$\sigma_{lgD}(M_{0})=\sigma_{lgD}(A)\cup\sigma_{lgD}(B)$$
\end{proposition}
\begin{proof}
$\lambda\in\sigma_{lgD}(M_0)$ if and only if $\lambda\in acc(\sigma_{ap}(M_0))$
 if and only if $\lambda\in acc(\sigma_{ap}(A)\cup\sigma_{ap}(B))=acc(\sigma_{ap}(A))\cup acc(\sigma_{ap}(B))$ if and only if $\lambda\in \sigma_{lgD}(A)\cup\sigma_{lgD}(B)$.
\end{proof}
By duality, we have:
\begin{proposition}
Let $(A, B)\in \mathcal{B}(X)\times \mathcal{B}(Y)$  and  $C\in\mathcal{B}(Y, X)$. Then :
$$\sigma_{rgD}(M_{0})=\sigma_{rgD}(A)\cup\sigma_{rgD}(B)$$
\end{proposition}
As a straightforward consequence, we have the result of H.Zariouh and H. Zghitti \cite{ZZ}.
\begin{corollary}\cite{ZZ}
Let $(A, B)\in \mathcal{B}(X)\times \mathcal{B}(Y)$  and  $C\in\mathcal{B}(Y, X)$. Then :
$$\sigma_{gD}(M_{0})=\sigma_{gD}(A)\cup\sigma_{gD}(B)$$
\end{corollary}

\begin{proposition}
Let $(A, B)\in \mathcal{B}(X)\times \mathcal{B}(Y)$  and  $C\in\mathcal{B}(Y, X)$.
\begin{enumerate}
  \item If $A$  is invertible, then $M_C$ is left generalized Drazin invertible if and only if $B$ is left generalized Drazin invertible.
  \item If $B$  is invertible, then $M_C$ is right  generalized Drazin invertible if and only if $A$ is right generalized Drazin invertible.
\end{enumerate}
\end{proposition}

\begin{proof}
$ 1)$ Suppose that $M_C$ is left generalized Drazin invertible, then $0\in acc (\sigma_{ap}(M_C))$,  thus there exists $\varepsilon> 0$ such that $M_C-\lambda I$ is bounded below  for every $\lambda, 0<|\lambda|<\varepsilon$. Since $A$ is invertible, then there exists $\beta> 0$ such that $A-\lambda I$ is invertible  for every $\lambda, |\lambda|<\beta$. Let $\eta=min(\varepsilon, \beta)$, then $A-\lambda I$ is invertible for every $\lambda, |\lambda|<\eta$ and $M_C-\lambda I$ is bounded below  for every $\lambda, 0<|\lambda|<\eta$. Hence $B-\lambda I$ is bounded below  for every $\lambda, 0<|\lambda|<\eta$, by \cite [Lemma 2.7]{ZZW}.\\
By duality, we have  $2)$.
\end{proof}

\begin{theorem}
Let $(A, B)\in \mathcal{B}(X)\times \mathcal{B}(Y)$  and  $C\in\mathcal{B}(Y, X)$. Then :
$$\sigma_{lgD}(A)\subseteq\sigma_{lgD}(M_C)\subseteq\sigma_{lgD}(A)\cup\sigma_{lgD}(B)\subseteq\sigma_{lgD}(M_C)\cup\sigma_{lgD}(B)$$
\end{theorem}
\begin{proof}
 Let $\lambda=0\notin \sigma_{lgD}(A)\cup\sigma_{lgD}(B)$, then $0\notin acc(\sigma_{ap}(A)\cup\sigma_{ap}(B))=acc(\sigma_{ap}(A))\cup acc(\sigma_{ap}(B))$. Thus there exists $\varepsilon> 0$ such that  for any $\lambda$, $0< |\lambda|< \varepsilon$, we have $A-\lambda I$ and $B-\lambda I$ are bounded below.  According to \cite[Theorem 3.5]{ZZZ}, we have  $M_C-\lambda I$  is bounded below for any $\lambda$, $0< |\lambda|< \varepsilon$, thus $0\notin \sigma_{lgD}(M_C)$. Therefore $\sigma_{lgD}(M_C)\subseteq \sigma_{lgD}(A)\cup\sigma_{lgD}(B)$.\\
 If $0\notin \sigma_{lgD}(M_C)$,  then $0\notin acc(\sigma_{ap}(M_C))$, as a result, there exists $\varepsilon> 0$ such that  for any $\lambda$, $0< |\lambda|< \varepsilon$, we have $M_C-\lambda I$ is bounded below, then $A-\lambda I$  is bounded below for any $\lambda$, $0< |\lambda|< \varepsilon$ by \cite[Theorem 3.5]{ZZZ}, thus $0\notin \sigma_{lgD}(A)$. Therefore $\sigma_{lgD}(A)\subseteq\sigma_{lgD}(M_C)$ \\
 Let $\lambda=0\notin \sigma_{lgD}(M_C)\cup\sigma_{lgD}(B)$, then $0\notin acc(\sigma_{ap}(M_C)\cup\sigma_{ap}(B))=acc(\sigma_{ap}(M_C))\cup acc(\sigma_{ap}(B))$. Thus there exists $\varepsilon> 0$ such that  for any $\lambda$, $0< |\lambda|< \varepsilon$, we have $M_C-\lambda I$ and $B-\lambda I$ are bounded below.  According to \cite[Theorem 3.5]{ZZZ}, we have  $A-\lambda I$  is bounded below for any $\lambda$, $0< |\lambda|< \varepsilon$, thus $0\notin \sigma_{lgD}(A)$. Therefore $\sigma_{lgD}(A)\cup\sigma_{lgD}(B)\subseteq\sigma_{lgD}(M_C)\cup\sigma_{lgD}(B)$.\\
 \end{proof}
 By duality, we can prove the following Theorem.
 \begin{theorem}
Let $(A, B)\in \mathcal{B}(X)\times \mathcal{B}(Y)$  and  $C\in\mathcal{B}(Y, X)$. Then :
$$\sigma_{rgD}(B)\subseteq\sigma_{rgD}(M_C)\subseteq\sigma_{rgD}(A)\cup\sigma_{rgD}(B)\subseteq\sigma_{rgD}(M_C)\cup\sigma_{rgD}(A)$$
\end{theorem}
The  inclusion,  $\sigma_{rgD}(M_C)\subseteq\sigma_{rgD}(A)\cup\sigma_{rgD}(B)$,  may be strict as we can see in the following example.
\begin{example}
Let $A, B, C\in\mathcal{B}(l_2)$ defined by:
$$ Ae_n=e_{n+1}.$$
$$ B=A^*.$$
$$ C=e_0\otimes e_0.$$
 where $\{e_n\}_{n\in\mathbb{N}}$ is the unit vector basis of $l_2$. We have $\sigma_{su}(A)=\{\lambda\in\mathbb{C}; |\lambda|\leq 1\}$, then $\sigma_{rgD}(A)=\{\lambda\in\mathbb{C}; |\lambda|\leq 1\}$. Since  $M_C$ is unitary,  then $\sigma_{rgD}(M_C)\subseteq\{\lambda\in\mathbb{C}; |\lambda|= 0\}$.  So  $0\notin\sigma_{rgD}(M_C)$, but $0\in\sigma_{rgD}(A)\cup\sigma_{rgD}(B)$. Notes that $A^*=B$ has not the SVEP.
\end{example}
\begin{remark}
Also,  we can show that the inclusion $\sigma_{lgD}(M_C)\subset\sigma_{lgD}(A)\cup\sigma_{lgD}(B)$ is strict.
\end{remark}

Now, in the next results, we give a sufficient conditions which ensures that  $\sigma_{rgD}(M_C)\subset\sigma_{rgD}(A)\cup\sigma_{rgD}(B)$ and
 $\sigma_{lgD}(M_C)\subset\sigma_{lgD}(A)\cup\sigma_{lgD}(B)$ for every $C\in\mathcal{B}(Y,X)$.
\begin{proposition}\label{ppp}
Let $A\in \mathcal{B}(X)$ such that $A^*$ has the SVEP, then for every $B\in \mathcal{B}(Y)$ and  $C\in\mathcal{B}(Y, X)$, we have:
$$\sigma_{lgD}(M_{C})=\sigma_{lgD}(A)\cup\sigma_{lgD}(B)$$
\end{proposition}
\begin{proof}
$\lambda\in\sigma_{lgD}(M_C)$ if and only if $\lambda\in acc(\sigma_{ap}(M_C))$, according to \cite[Corollary 3.13]{ZZZ}, we have $\lambda\in acc(\sigma_{ap}(A)\cup\sigma_{ap}(B))$, since $acc(\sigma_{ap}(A)\cup\sigma_{ap}(B))=acc(\sigma_{ap}(A))\cup acc(\sigma_{ap}(B))$, then $\lambda\in acc(\sigma_{ap}(M_C))$ if and only if $\lambda\in acc(\sigma_{ap}(A))\cup acc(\sigma_{ap}(B))$ if and only if $\lambda\in \sigma_{lgD}(A)\cup\sigma_{lgD}(B)$.
\end{proof}
By duality, we have :
\begin{proposition}\label{pppp}
Let $B\in \mathcal{B}(X)$ such that $B$ has the SVEP, then for every $A\in \mathcal{B}(X)$ and $C\in\mathcal{B}(Y, X)$, we have:
$$\sigma_{rgD}(M_{C})=\sigma_{rgD}(A)\cup\sigma_{rgD}(B)$$
\end{proposition}
From  Proposition \ref{ppp} and Proposition \ref{pppp}, we have
\begin{corollary}\cite{ZZ}
Let $(A,B)\in \mathcal{B}(X) \times\mathcal{B}(Y)$ such that $A^*$ and  $B$ have  the SVEP, then for every $C\in\mathcal{B}(Y, X)$, we have:
$$\sigma_{gD}(M_{C})=\sigma_{gD}(A)\cup\sigma_{gD}(B)$$
\end{corollary}

From \cite[Proposition 3.1]{HZ}, if $A$ and $B$ have the SVEP, then $M_C$ has also SVEP. The following corollary hold.
\begin{corollary}
Let $A\in\mathcal{B}(X)$, $B\in\mathcal{B}(Y)$.
\begin{enumerate}
  \item If $A, A^*$ and $B$ have the SVEP, then for every $C\in\mathcal{B}(Y, X)$
  $$\sigma_{pBuF}(M_{C})=\sigma_{pBuF}(A)\cup\sigma_{pBuF}(B)$$
  \item If $B, B^*$ and $A$ have the SVEP, then for every $C\in\mathcal{B}(Y, X)$
  $$\sigma_{pBlF}(M_{C})=\sigma_{pBlF}(A)\cup\sigma_{pBlF}(B)$$
  \item If $A, A^*$  $B,$ and $B^*$ have the SVEP, then for every $C\in\mathcal{B}(Y, X)$
  $$\sigma_{pBF}(M_{C})=\sigma_{pBF}(A)\cup\sigma_{pBF}(B)$$
\end{enumerate}
\end{corollary}

\begin{corollary}
Let $(A, B)\in \mathcal{B}(X)\times \mathcal{B}(Y)$ such that $A$  and $B$ are decomposable, then for every  $C\in\mathcal{B}(Y, X)$, we have:
$$\sigma_{pBF}(M_{C})=\sigma_{pBF}(A)\cup\sigma_{pBF}(B)$$
\end{corollary}


\end{document}